\documentclass[reqno]{amsart}
\usepackage{amsmath,amssymb,amsthm,url}
\usepackage{graphicx,a4wide}

\theoremstyle{plain}
\newtheorem{thm}{Theorem}
\theoremstyle{definition}
\newtheorem{defn}{Definition}
\theoremstyle{remark}
\newtheorem{rem}{Remark}
\newtheorem{ex}{Example}

\newcommand{\prn}[1]{\left(#1\right)}

\newcommand{\norm}[1]{\left\|#1\right\|}
\newcommand{\pd}[2]{\frac{\partial#1}{\partial#2}}

\begin{document}
\parindent0ex
\parskip2ex

\title[Performance of AMG for Non-Symmetric Matrices]
{Performance of Algebraic Multigrid Methods
for Non-Symmetric Matrices arising in Particle Methods}
\author{Benjamin Seibold}
\address{Department of Mathematics \\ Temple University \\
1801 N.~Broad Street \\ Philadelphia, PA 19122}
\email{seibold@temple.edu}
\urladdr{http://www.math.temple.edu/\~{ }seibold}
\begin{abstract}
Large linear systems with sparse, non-symmetric matrices arise in the modeling of
Markov chains or in the discretization of convection-diffusion problems. Due to their
potential to solve sparse linear systems with an effort that is linear in the number
of unknowns, algebraic multigrid (AMG) methods are of fundamental interest for such
systems. For symmetric positive definite matrices, fundamental theoretical convergence
results are established, and efficient AMG solvers have been developed.
In contrast, for non-symmetric matrices, theoretical convergence results have been
provided only recently. A property that is sufficient for convergence is that the
matrix be an M-matrix.
In this paper, we present how the simulation of incompressible fluid flows with
particle methods leads to large linear systems with sparse, non-symmetric matrices.
In each time step, the Poisson equation is approximated by meshfree finite differences.
While traditional least squares approaches do not guarantee an M-matrix structure,
an approach based on linear optimization yields optimally sparse M-matrices.
For both types of discretization approaches, we investigate the performance of a
classical AMG method, as well as an AMLI type method.
While in the considered test problems, the M-matrix structure turns out not to be
necessary for the convergence of AMG, problems can occur when it is violated. In
addition, the matrices obtained by the linear optimization approach result in fast
solution times due to their optimal sparsity.
\end{abstract}

\maketitle

\section{Introduction}
The efficient and robust numerical approximation of the Poisson equation is a key
problem in many applications. In some special cases (such as spectral methods that
use the fast Fourier transform), the Poisson equation can be solved in a single
step. However, in most scenarios, the solution process is done in two consecutive steps:
first, a discretization transforms the Poisson equation into a finite linear system;
second, the linear system is solved. While in most investigations on multigrid methods
the discretization is assumed as canonically given, in this paper we present an
application in which many possible discretization strategies exist, and the interplay
between discretization and performance of multigrid solvers is of interest.

We consider Poisson problems that arise in Lagrangian particle methods for the
simulation of incompressible fluid flows. Particle methods solve the governing fluid
equations on a set of nodes that move with the flow. Thus the discretization is done
in the more natural Lagrangian frame of reference. As a consequence, convective terms
are solved exactly. One of the earliest Lagrangian particle methods
is \emph{smoothed particle hydrodynamics} (SPH) \cite{Lucy1977,GingoldMonaghan1977},
which was first applied to astrophysical fluid dynamics, and later expanded to a
wider class of hydrodynamic equations \cite{Monaghan1988,Monaghan1992}.
Generalization of the SPH approach, that use moving least squares approximations for
derivatives, have been presented by Dilts \cite{Dilts1999}
and Kuhnert \cite{KuhnertDiss1999}. Both approaches have been successfully applied
to problems with complex, time-dependent geometries \cite{KuhnertTiwari2002} and
free surface flows \cite{KuhnertTiwariPasadena2003}. Further examples are shown
in \cite{SeiboldMeshfree2007}. The problems considered in this paper are mostly
formulated in the terminology of those latter two approaches. However, many aspects
transfer to other types of particle methods.
A fundamental difficulty with Lagrangian particle methods is that one gives away
the control over the positions of the computational nodes. While one can (and in
fact should) remove particles where not desired, and insert new particles where
required, the points are in general completely unstructured and without connections
between each other.

A classical solution strategy for incompressible flows is the Chorin projection
method \cite{Chorin1968}, which has been successfully applied to Lagrangian particle
methods \cite{CumminsRudman1998,CumminsRudman1999}.
In each time step, the flow is first evolved neglecting
pressure, which yields a velocity field that is not incompressible. Then the
velocity field is modified by a gradient field that is chosen so that
incompressibility is restored. This correction step requires the solution of a Poisson
equation. As noted above, the arrangement of the particles is in a general,
unstructured cloud of points. And the Poisson equation has to be approximated on
this geometry.

One possibility to approximate the Poisson equation is using a finite element approach.
For instance, one can generate a triangulation \cite{Shewchuk1999}
(or another type of mesh) between the points, and apply a traditional finite element
approach \cite{StrangFix1973}.
However, mesh generation is an expensive task, especially in 3d, and a good mesh
generator may not always be available. In the context of particle methods, one
rarely is willing to make such an investment.
An alternative is the use of meshfree finite element approaches, such as
the \emph{element-free Galerkin method} (EFGM) \cite{BelytschkoGuLu1994},
or the \emph{partition of unity finite element method}
(PUFEM) \cite{BabuskaMelenk1996,BabuskaMelenk1997}.
Similar approaches \cite{GriebelSchweitzer2000} have been applied in a
multilevel context \cite{GriebelSchweitzer2002_2}.
A key advantage of all these approaches is that they are based on inner products
between (meshfree) basis functions, hence they lead to symmetric linear systems
to be solved. A fundamental drawback is that the computation of the inner products
requires quadratures, which in many cases can only be achieved by introducing a regular
background grid. This step is not only technically inconvenient, it is also costly.
Another difficulty with finite element approaches is the treatment of general boundary
conditions, which arise in multi-phase flows in complex, time-dependent domains.

In contrast, meshfree finite difference methods do not require basis functions
or quadrature, and boundary conditions can be implemented directly. At each point,
the approximation of the Laplace operator is based on a local Taylor expansion.
A first approach by Perrone and Kao \cite{PerroneKao1975} was later extended by
Liszka, Orkisz, et al.~\cite{LiszkaOrkisz1980,DemkowiczKarafiatLiszka1984} to
the \emph{generalized finite difference method} (GFDM).
A detailed presentation of common particle methods and meshfree approaches is provided
in \cite{SeiboldDiss2006}.
In this paper, we outline the fundamental conditions for a consistent approximation
of the Laplace operator in Sect.~\ref{seibold:sec:fd_poisson}.

A key drawback of meshfree finite difference approaches is that the resulting linear
systems are in general non-symmetric. As we explain in
Sect.~\ref{seibold:sec:non_symmetry_AMG}, this non-symmetry cannot be overcome by any
simple means. In the same section, we outline theoretical results on the convergence of
algebraic multigrid approaches for non-symmetric matrices. A property sufficient for
convergence is an M-matrix structure. We present fundamental properties of M-matrices
and their relevance for multilevel methods in Sect.~\ref{seibold:sec:m_matrices}.

In addition to the non-symmetry of the meshfree finite difference matrices
classical least squares approaches, as outlined in
Sect.~\ref{seibold:sec:least_squares_vs_linear_minimization},
involve two other difficulties.
First, the arising matrices are about six times as dense as finite difference
matrices on rectangular grids.
Second, an M-matrix structure is not guaranteed, and is---unless the point geometry
is sufficiently nice---in general violated.
Unlike the non-symmetry, these latter two problems can be overcome,
by a new approach that is based on linear sign-constrained optimization.
This approach is presented in Sect.~\ref{seibold:sec:least_squares_vs_linear_minimization},
and conditions for its successful application are provided in
Sect.~\ref{seibold:sec:lp_solution_existence_connectivity}.
In Sect.~\ref{seibold:sec:effort_generation}, theoretical predictions about the computational
cost to generate the Poisson matrices with least squares vs.~linear optimization
approaches are given.
In Sect.~\ref{seibold:sec:numerics}, we numerically investigate the performance of
various multigrid approaches for the matrices constructed by the different types of
minimization approaches.
Conclusions and an outlook are given in Sect.~\ref{seibold:sec:conclusions_outlook}.

\section{Meshfree Finite Differences for the Poisson Equation}
\label{seibold:sec:fd_poisson}
Consider the Poisson equation to be solved inside a domain
$\Omega\subset\mathbb{R}^d$
\begin{equation}
\begin{cases}
  -\Delta u = f &\mathrm{in~}\Omega \\
\quad\;\; u = g &\mathrm{on~}\Gamma_D \\
\quad \frac{\partial u}{\partial n} = h &\mathrm{on~}\Gamma_N
\end{cases}
\label{seibold:eq:poisson_equation}
\end{equation}
where $\Gamma_D\cup\Gamma_N=\partial\Omega$. We consider the case
of \eqref{seibold:eq:poisson_equation} possessing a unique smooth solution.

Let a point cloud $X=\{\vec{x}_1,\dots,\vec{x}_n\}\subset\overline{\Omega}$ be given,
which consists of interior points $X_i\subset\Omega$ and boundary points
$X_b\subset\partial\Omega$. The point cloud is meshfree, i.e.~no information about
connection of points is provided. Meshfree finite difference approaches convert
problem \eqref{seibold:eq:poisson_equation} into a linear system
\begin{equation}
A\cdot\vec{\hat u} = \vec{\hat f}\;,
\label{seibold:eq:poisson_system}
\end{equation}
where the vector $\vec{\hat u}$ contains approximations to the values $u(\vec{x}_i)$.
The $i$-th row of the matrix $A$ consists of the \emph{stencil} corresponding to the
point $\vec{x}_i$.

In the following, we derive conditions on finite difference stencils to yield an
approximation to the Laplace operator.
Consider a function $u\in C^2(\Omega\subset\mathbb{R}^d,\mathbb{R})$. At each interior
point $\vec{x}_0$, we wish to approximate $\Delta u(\vec{x}_0)$ using the function
values of a finite number of points in a circular neighborhood
$(\vec{x}_0,\vec{x}_1,\dots,\vec{x}_m)\in B(\vec{x}_0,r)$, where
$B(\vec{x}_0,r) = \{\vec{x}\in\overline{\Omega}:\|\vec{x}-\vec{x}_0\|_2<r\}$.
Note that the Euclidian norm is a natural choice in meshfree methods due to its
rotational symmetry. Define the distance vectors
$\vec{\bar x}_i = \vec{x}_i-\vec{x}_0 \ \forall i=0,\dots,m$.
The function value at each neighboring point $u(\vec{x}_i)$ is given by a Taylor
expansion as
\begin{equation*}
u(\vec{x}_i) = u(\vec{x}_0)+\nabla u(\vec{x}_0)\cdot\vec{\bar x}_i
+\tfrac{1}{2}\nabla^2 u(\vec{x}_0):\prn{\vec{\bar x}_i\cdot\vec{\bar x}_i^T}+e_i\;.
\end{equation*}
Here we use the Frobenius inner product $A:B=\sum_{i,j}a_{ij}b_{ij}$. The error in
the expansion is of order $e_i = O(r^3)$. A linear combination with coefficients
$(s_0,\dots,s_m)$ equals
\begin{align*}
\sum_{i=0}^m s_i u(\vec{x}_i) =
u(\vec{x}_0)\prn{\sum_{i=0}^m s_i}
+\nabla u(\vec{x}_0)\cdot\prn{\sum_{i=1}^m s_i\vec{\bar x}_i}\;\; \\
+\nabla^2 u(\vec{x}_0):\prn{\frac{1}{2}\sum_{i=1}^m
s_i\prn{\vec{\bar x}_i\cdot\vec{\bar x}_i^T}}
+\prn{\sum_{i=1}^m s_ie_i}\;.
\end{align*}
This approximates the Laplacian,
i.e.~$\sum_{i=0}^m s_i u(\vec{x}_i) = \Delta u(\vec{x}_0)+O(r^3)$,
if exactness for constant, linear, and quadratic functions is satisfied
\begin{equation}
\sum_{i=0}^m s_i = 0 \quad , \quad
\sum_{i=1}^m\vec{\bar x}_i s_i = 0 \quad , \quad
\sum_{i=1}^m\prn{\vec{\bar x}_i\cdot\vec{\bar x}_i^T}s_i = 2I\;.
\label{seibold:eq:constraints_laplace}
\end{equation}
\begin{defn}
\label{seibold:def:consistency}
A stencil $(s_0,\dots,s_m)$ to a set of points
$(\vec{x}_0,\vec{x}_1,\dots,\vec{x}_m)\in B(\vec{x}_0,r)$ is called \emph{consistent}
(with the Laplace operator), if the constraints \eqref{seibold:eq:constraints_laplace}
are satisfied.
\end{defn}
The linear and quadratic constraints in \eqref{seibold:eq:constraints_laplace} can be
formulated as a linear system of equations
\begin{equation}
V\cdot\vec{s} = \vec{b}\;,
\label{seibold:eq:linear_system}
\end{equation}
where $V\in\mathbb{R}^{k\times m}$ is the Vandermonde matrix given
by $\vec{\bar x}_1,\dots,\vec{\bar x}_m$, and $\vec{s}\in\mathbb{R}^m$ is the
stencil vector.
The number of constraints is $k = \frac{d(d+3)}{2}$.
The constant constraint in \eqref{seibold:eq:constraints_laplace} yields
$s_0 = -\sum_{i=1}^m s_i$.
Neumann boundary points can be treated in a similar manner.
Approximating $\pd{u}{\vec{n}}(\vec{x}_0)$ by $\sum_{i=0}^m s_i u(\vec{x}_i)$
leads to the constraints
\begin{equation}
\sum_{i=0}^m s_i = 0 \quad , \quad
\sum_{i=1}^m\vec{\bar x}_i s_i = \vec{n}\;.
\label{seibold:eq:constraints_Neumann}
\end{equation}
For each point, a meshfree finite difference approximation consists of two steps:
First, define which points are its neighbors. Typically, more neighbors than
constraints are chosen. Second, select a stencil.
If \eqref{seibold:eq:constraints_laplace} is underdetermined, a minimization problem
is formulated to select a unique stencil.
The results in Sect.~\ref{seibold:sec:lp_solution_existence_connectivity} guarantee
the existence of a solution to the above linear
systems \eqref{seibold:eq:constraints_laplace}
and \eqref{seibold:eq:constraints_Neumann}.

\begin{defn}
\label{seibold:def:minimal_stencil}
A consistent stencil $(s_0,\dots,s_m)$ is called \emph{minimal}, if $m\le k$.
\end{defn}
Minimal stencils are beneficial for the sparsity of the system matrix, resulting in a
minimal memory consumption and a fast application of the matrix. In fact, the total
number of neighboring points is proportional to the effort of multiplying the matrix
with a vector. This promises a fast solution of system \eqref{seibold:eq:poisson_system}
with (semi-)iterative linear solvers, assumed the convergence properties are reasonable.
More importantly, having minimal stencils results in an optimally sparse matrix graph.
This property can be advantageous with respect to coarsening in algebraic
multigrid methods.

\begin{defn}
\label{seibold:def:positive_stencil}
A consistent stencil $(s_0,\dots,s_m)$ is called \emph{positive}, if \linebreak
$s_1,\dots,s_m\ge 0$. Due to \eqref{seibold:eq:constraints_laplace} and
\eqref{seibold:eq:constraints_Neumann}, this implies for the central point $s_0<0$.
\end{defn}
Having positive stencils implies that the system matrix
in \eqref{seibold:eq:poisson_system} is an L-matrix (Def.~\ref{seibold:def:L_matrix}),
which gives rise to an M-matrix structure (see Sect.~\ref{seibold:sec:m_matrices}).
The desirability of positive stencils has been pointed out
by Demkowicz, Karafiat and Liszka \cite{DemkowiczKarafiatLiszka1984}.
Classical approaches do in general not yield positive stencils.
An ``optimal star selection'' \cite{DuarteLiszkaTworzyako1996} makes positive stencils
likely, but they are not guaranteed.
In Sect.~\ref{seibold:sec:least_squares_vs_linear_minimization} we present a strategy
that approximates the Poisson equation \eqref{seibold:eq:poisson_equation} on a point
cloud by minimal positive stencils.
In Sect.~\ref{seibold:sec:lp_solution_existence_connectivity} conditions on a point
cloud are given which guarantee the existence of positive stencils.

\section{Non-Symmetric Matrices and Algebraic Multigrid}
\label{seibold:sec:non_symmetry_AMG}
Meshfree finite difference approaches, as described in Sect.~\ref{seibold:sec:fd_poisson},
approximate the Poisson equation by non-symmetric matrices.
Large linear systems with non-symmetric matrices have received considerable attention
in Markov chain modeling. For instance, the computation of a page ranking in internet
search engines leads to sparse non-symmetric matrices \cite{LangvilleMeyer2006}.
The arising matrices typically have an M-matrix structure.
In the approximation of partial differential equations, non-symmetric matrices
traditionally arise when convective terms are approximated.
As described in \cite{TrottenbergOosterleeSchuller2001}, classical AMG methods
frequently are observed to converge for such non-symmetric matrices, if the advection
does not dominate the problem. When approximating the Poisson equation, traditional
approaches lead to symmetric matrices. In contrast, meshfree finite difference approaches
approximate the Poisson equation by non-symmetric matrices. At first glance this
sounds surprising, since the Laplace operator itself is selfadjoint. However, it can
be proved and explained.

The formal argument for non-symmetry is as follows.
Consider two points $\vec{x}_i$ and $\vec{x}_j$, each being a neighbor of the other,
and a third point $\vec{x}_k$ which is a neighbor of $\vec{x}_i$ but not a neighbor of
$\vec{x}_j$. Since each stencil entry depends on \emph{all} its neighbors, the point
$\vec{x}_k$ influences the matrix entry $a_{ij}$, but not the matrix entry $a_{ji}$.

The non-symmetry of the matrix can also be understood in that a self-adjoint operator
is approximated on a discrete set of points that is unstructured, and thus does not
possess any local symmetry. Note that on unstructured geometries, classical finite
element approaches \cite{StrangFix1973} approximate the Laplacian by a non-symmetric
matrix as well. While in the finite element system $L\,\vec{u} = M\vec{f}$,
the stiffness matrix $L$ alone is always symmetric, the actual approximation matrix
$M^{-1}L$ is in general non-symmetric.

The non-symmetry is inherent in the finite difference approximation, and to the
author's knowledge, there is no straightforward way to symmetrize the linear system
\eqref{seibold:eq:poisson_system}. This has interesting implications for the choice
of efficient solvers for the arising linear systems. For instance, in the realm of
(preconditioned) Krylov subspace methods, the classical conjugate gradient method
has to be replaced by more costly solvers, such as GMRES \cite{SaadSchultz1986}
or BiCGstab \cite{VanDerVorst1992}.

Of fundamental interest are multigrid solvers for the arising
systems \eqref{seibold:eq:poisson_system}. In most practical applications with a
sufficient number of particles, the largest part of the computational time is spent
on the generation and solution of the linear systems in the pressure correction
(and possibly in an implicit treatment of viscosity, which leads to quite similar
system matrices). Hence it is of crucial interest to achieve the optimal
computational effort $O(n)$, where $n$ is the number of particles, that is promised
by multigrid methods.

In the context of particle methods, \emph{algebraic multigrid} (AMG) approaches are most
promising, since the application of geometric multigrid approaches on an unstructured,
non-connected cloud of points is very difficult. While in the symmetric matrix case,
various convergence results for AMG are well
known \cite{Brandt1986,Stueben2001,TrottenbergOosterleeSchuller2001},
theoretical results for the non-symmetric matrix are much fewer in number.
An analysis of AMG for non-symmetric matrices has been given by
Notay \cite{Notay2000,Notay2009}.
A complementary analysis for methods of \emph{algebraic multilevel iteration} (AMLI)
type \cite{AxelssonVassilevski1989,AxelssonVassilevski1990}
has been given by Mense and Nabben \cite{MenseNabben2008_1,MenseNabben2008_2}.
AMLI methods are based on a block incomplete factorization of the matrix, partitioned
in hierarchical form. Many AMG approaches can be interpreted as AMLI methods.

A common assumption in all the above theoretical results is that the matrix be an
M-matrix (see Sect.~\ref{seibold:sec:m_matrices}). The M-matrix structure implies
certain properties that positive definiteness would imply for symmetric matrices.
Since the negative Laplace operator in \eqref{seibold:eq:poisson_equation} is
positive definite, the M-matrix structure is natural for approximations of it.
In above references, proofs of convergence of two-grid methods for the M-matrix case
are given; some results can be extended to the full multigrid case. In particular,
it yields a discrete maximum principle, and guarantees the convergence of the
Gau{\ss}-Seidel iteration (see Sect.~\ref{seibold:subsec:mmatrix_benefits}).
In the above referenced theoretical results, the M-matrix structure is sufficient
for the convergence of AMG and AMLI methods. The importance of the M-matrix
structure for AMG is discussed in \cite{Notay2000}. The analysis of the two-grid
case is given in \cite{Notay2009}. A key step in the generalization of theoretical
AMG results to the non-symmetric matrix case is the suitable understanding of the
size of an eigenvalue as its modulus in the complex case.
In Sect.~\ref{seibold:subsec:mmatrix_multilevel}, recent theoretical results on
the convergence of AMLI methods in the M-matrix case are outlined.

Of course, the M-matrix property is not necessary for the convergence of multigrid
approaches. In fact, convergence is frequently observed even if the M-matrix structure
is violated. However, unless theoretical results for weaker assumptions are present,
one cannot be sure.
This would not be too problematic if a single Poisson equation had to be
solved. However, if the Poisson solves are just substeps in a larger computation
(as it is the case in the particle method), it is unsatisfactory to not have a
guarantee of convergence. In fact, one may observe a rapid AMG convergence in the
beginning of a long computation (when the particles are distributed nicely), but
convergence may deteriorate as the particles become less evenly distributed.
It is our hope that for meshfree finite difference methods, the M-matrix property
can be a key instrument in successful and efficient solution approaches for the
Poisson equation. In addition, it would be desirable to obtain an understanding of
the connection between the performance of AMG and the corresponding particle geometry.

\section{Properties of M-Matrices}
\label{seibold:sec:m_matrices}
In this section, we provide a brief overview of the definition of the M-matrix property,
a practical sufficient condition that relates to meshfree finite difference
approximations, and some well-known benefits of the M-matrix structure.
\begin{defn}
\label{seibold:def:L_matrix}
A square matrix $A=(a_{ij})_{ij}\in\mathbb{R}^{n\times n}$ is called \emph{Z-matrix},
if $a_{ij}\le 0 \ \forall i\neq j$.
A Z-matrix is called \emph{L-matrix}, if $a_{ii}>0 \ \forall i$.
\end{defn}
We write $A\ge 0$ if $a_{ij}\ge 0 \ \forall i,j$. The same notation applies to vectors.
\begin{defn}
A regular matrix $A$ is called \emph{inverse positive}, if $A^{-1}\ge 0$.
\end{defn}
\begin{defn}
A Z-matrix is called \emph{M-matrix}, if it is inverse positive.
\end{defn}
We use the M-matrix property, since it yields a sufficient condition for inverse
positivity.

\subsection{A Sufficient Condition for the M-Matrix Structure}
\label{seibold:subsec:mmatrix_suff_cond}
Conditions that imply the M-matrix property are required, since the inverse matrix is
typically not directly available.
Let the unknowns be labeled by an index set $I$. We consider square matrices
$A\in\mathbb{R}^{I\times I}$.
\begin{defn}
The \emph{graph} $G(A)$ of a matrix $A$ is defined by
$G(A) = \{(i,j)\in I\times I : a_{ij}\neq 0\}$.
The index $i\in I$ is called \emph{connected} to $j\in I$, if a chain
$i=i_0,i_1,\dots,i_{k-1},i_k=j \ \in I$ exists, such that
$(i_{\nu-1},i_\nu)\in G(A) \ \forall \nu=1,\dots,k$.
\end{defn}
\begin{rem}
For a finite difference matrix, each index $i\in I$ corresponds to a point $\vec{x}_i$.
The index $i\in I$ being connected to $j\in I$ means that the point $\vec{x}_i$
connects (indirectly) to the point $\vec{x}_j$ via stencil entries.
\end{rem}
\begin{defn}
A finite difference matrix is called \emph{essentially irreducible} if every point
is connected to a Dirichlet boundary point.
\end{defn}
\begin{rem}
A finite difference matrix that is not essentially irreducible is singular,
since the points that are not connected to a Dirichlet point form a singular submatrix.
\end{rem}
\begin{defn}
A matrix $A\in\mathbb{R}^{I\times I}$ is called \emph{essentially diagonally dominant}
if it is weakly diagonally dominant ($\forall i\in I: |a_{ii}|\ge\sum_{k\neq i}|a_{ik}|$),
and every point $i\in I$ is connected to a point $j\in I$ which satisfies the strict
diagonal dominance relation $|a_{jj}|>\sum_{k\neq j}|a_{jk}|$.
\end{defn}
\begin{thm}
\label{seibold:thm:ess_irred_is_ess_diagdom}
An L-matrix arising as a finite difference discretization of
\eqref{seibold:eq:poisson_equation} is essentially diagonally dominant if it is
essentially irreducible.
\end{thm}
\begin{proof}
For an L-matrix the constant relation in \eqref{seibold:eq:constraints_laplace}
implies the weak diagonal dominance relation for every interior and Neumann point.
Each row corresponding to a Dirichlet point satisfies the strict diagonal dominance
relation.
\end{proof}
\begin{thm}
\label{seibold:thm:Lmatrix_is_Mmatrix}
An essentially diagonally dominant L-matrix is an M-matrix.
\end{thm}
\begin{proof}
The proof is given in \cite[p.~153]{Hackbusch1994}.
\end{proof}
If problem \eqref{seibold:eq:poisson_equation} can be discretized by positive stencils
and every point is connected to a Dirichlet point, then the resulting matrix
is an M-matrix.

\subsection{Some Benefits of the M-Matrix Structure}
\label{seibold:subsec:mmatrix_benefits}
The Poisson equation satisfies maximum principles. For instance, consider
\eqref{seibold:eq:poisson_equation} with Dirichlet boundary conditions only.
If $f\le 0$ and $g\le 0$, then the solution satisfies $u\le 0$ \cite{Evans1998}.
A discretization by an M-matrix mimics this property in a discrete maximum
principle.
\begin{thm}
\label{seibold:thm:mmatrix_discrete_max_principle}
Let $A$ be an M-matrix. Then $A\vec{x}\le 0$ implies $\vec{x}\le 0$.
Conversely, a Z-matrix satisfying $A\vec{x}\le 0\Rightarrow\vec{x}\le 0$ is an
M-matrix.
\end{thm}
\begin{proof}
A is an M-matrix, thus $A^{-1}\ge 0$ by definition. Let $\vec{y} = A\vec{x}$.
Then $\vec{x} = A^{-1}\vec{y}$. The component-wise inequalities $A^{-1}\ge 0$
and $\vec{y}\le 0$ imply $\vec{x}\le 0$. The reverse statement is proved in
\cite[p.~29]{QuarteroniSaccoSaleri2000}.
\end{proof}
\begin{thm}
If $A$ is an M-matrix, and $D$ its diagonal part, then $\rho(I-D^{-1}A)<1$,
thus the Jacobi and the Gau{\ss}-Seidel iteration converge.
\end{thm}
\begin{proof}
The convergence of the Jacobi iteration is given in \cite{Hackbusch1994}.
The Gau{\ss}-Seidel convergence follows from the
Stein-Rosenberg-Theorem \cite{Varga2000}.
\end{proof}
Further classical benefits of an M-matrix structure with respect to linear solvers
can be found in \cite{Varga2000}.

\subsection{Relevance of the M-Matrix Structure for Multilevel Methods}
\label{seibold:subsec:mmatrix_multilevel}
Recent theoretical convergence results for AMG in the non-symmetric M-matrix case have
been given in \cite{Notay2000,Notay2009}, and for AMLI type approaches in
\cite{MenseNabben2008_1,MenseNabben2008_2}. Here, we outline the key results for the
latter type of approaches. The M-matrix structure is a sufficient condition for the
convergence of additive Schwarz (AMLI) methods for non-symmetric matrices.
Consider the system matrix be partitioned into blocks
\begin{equation*}
A = \begin{bmatrix} A_{FF} & A_{FC} \\ A_{CF} & A_{CC} \end{bmatrix}\;,
\end{equation*}
where $F$ denotes the fine grid unknowns, and $C$ the coarse grid unknowns.
If $A_{FF}$ is regular, then $A$ can be factorized as
\begin{equation}
A = \begin{bmatrix} I_F & 0 \\ A_{CF}A_{FF}^{-1} & I_C \end{bmatrix}\cdot
\begin{bmatrix} A_{FF} & 0 \\ 0 & S \end{bmatrix}\cdot
\begin{bmatrix} I_F & A_{FF}^{-1}A_{FC} \\ 0 & I_C \end{bmatrix}\;,
\label{seibold:eq:Schwarz_factorization}
\end{equation}
where
\begin{equation}
S = A_{CC}-A_{CF}A_{FF}^{-1}A_{FC}
\label{eq:Schur_complement}
\end{equation}
is the Schur complement.
An approximation $M\approx A$ is obtained by multiplying \eqref{seibold:eq:Schwarz_factorization}
with the approximations $\tilde{A}_{FF}\approx A_{FF}$ and $\tilde{S}\approx S$.
A simple example for $\tilde{S}$ is to replace $A_{FF}$ by $\tilde{A}_{FF}$
in \eqref{eq:Schur_complement}.
As outlined in \cite{MenseNabben2008_2}, these block Schwarz factorization
approaches are closely related to two-grid AMG methods.
The AMLI method can be described as the stationary iteration with the iteration
matrix
\begin{equation*}
T = I-M^{-1}A\;.
\end{equation*}
\begin{defn}
A splitting $(M,N)$ of $A = M-N$ with $M$ regular is called
\emph{weak regular of the first type}, if $M^{-1}\ge 0$ and $M^{-1}N\ge 0$.
\end{defn}
The convergence of the AMLI method is ensured by the
\begin{thm}
\label{seibold:thm:AMLI_convergence}
If the splittings
$\prn{\tilde{A}_{FF},\tilde{A}_{FF}-A_{FF}}$ and $\prn{\tilde{S},\tilde{S}-S}$
are weak regular of the first type,
then $T\ge 0$, and $\rho(T)<1$.
\end{thm}
It is shown in \cite{MenseNabben2008_2} that if $A$ is an M-matrix, then the assumptions
in Thm.~\ref{seibold:thm:AMLI_convergence} are satisfied for approximations $\tilde{A}_{FF}$
given by the Jacobi and the Gauss-Seidel methods and the incomplete LU factorization.
However, the convergence proof fails in general, if $A$ does not have an M-matrix
structure.
In \cite{MenseNabben2008_2}, further methods have been presented (MAMLI, SMAMLI), that
are based on other types of Schwarz factorizations. Under similar types of assumptions
on the splittings, one has the convergence result
$\rho(T_{SMAMLI})\le\rho(T_{MAMLI})\le\rho(T_{AMLI})$.

\section{Least Squares vs.~Linear Minimization Approach}
\label{seibold:sec:least_squares_vs_linear_minimization}
Classical approaches for meshfree derivative approximation are moving least squares
methods, based on scattered data interpolation \cite{LancasterSalkauskas1981},
and local approximation methods, based on generalized finite difference
methods \cite{LiszkaOrkisz1980}. Their application to meshfree settings has
been analyzed in \cite{DuarteLiszkaTworzyako1996,Levin1998}.

Around a central point $\vec{x}_0$, points inside a radius $r$ are considered.
A distance weight function $w(\delta)$ is defined, which is small for $\delta>r$.
For instance, one can consider $w(\delta)=\delta^{-\alpha}$, where $\alpha\ge 1$.
Each point in the neighborhood $\vec{x}_i\in B(\vec{x}_0,r)$ is assigned a weight
$w_i = w\prn{\|\vec{x}_i-\vec{x}_0\|_2}$. A unique stencil is defined via a quadratic
minimization problem
\begin{equation}
\min\sum_{i=1}^n\frac{s_i^2}{w_i}, \ \mathrm{s.t.} \
V\cdot\vec{s} = \vec{b}\;.
\label{seibold:eq:quadratic_minimization}
\end{equation}
Using $W=\mathrm{diag}(w_i)$, the solution of \eqref{seibold:eq:quadratic_minimization}
is given by
\begin{equation}
\vec{s} = WV^T(VWV^T)^{-1}\cdot\vec{b}\;.
\label{seibold:eq:stencil_lsq}
\end{equation}
Least squares approaches do not yield minimal stencils, unless exactly $k$ neighbors
are considered. In general, they also do not yield positive stencils, as the following
example shows.

\begin{ex}
\label{seibold:ex:qm_nonpos_stencil}
Consider $\vec{x}_0=(0,0)$ and 6 neighbors on the unit
circle $\vec{x}_i=(\cos(\frac{\pi}{2}\varphi_i),\sin(\frac{\pi}{2}\varphi_i))$,
where $(\varphi_1,\dots,\varphi_6)=(0,1,2,3,0.1,0.2)$.
Since all neighbors have the same distance from $\vec{x}_0$, the distance weight
function does not play a role. Formula \eqref{seibold:eq:stencil_lsq} yields the
non-positive least squares stencil $\vec{s}=(0.846,1.005,0.998,1.003,0.312,-0.164)$.
However, the configuration admits a positive stencil, namely $\vec{s}=(1,1,1,1,0,0)$.
\end{ex}
As outlined in Sects.~\ref{seibold:sec:non_symmetry_AMG}
and \ref{seibold:subsec:mmatrix_multilevel}, we can be sure that multigrid
approaches converge, if an M-matrix structure is generated. Hence, as a new approach,
we enforce the positivity of stencils, i.e.~we search for solutions in the polyhedron
\begin{equation}
P = \{\vec{s}\in\mathbb{R}^m : V\cdot\vec{s} = \vec{b} , \
\vec{s}\ge 0\}\;.
\label{seibold:eq:polyhedron}
\end{equation}
This is the feasibility problem of linear optimization \cite{Vanderbei2001}.
In Sect.~\ref{seibold:sec:lp_solution_existence_connectivity} we provide conditions for
the point geometry that guarantee the existence of solutions.
If $P$ is non-void and not degenerate, there are infinitely many feasible stencils.
To single out a unique stencil we now formulate a \emph{linear} (or $\ell^1$)
minimization problem
\begin{equation}
\min \sum_{i=1}^m \frac{s_i}{w_i}, \ \mathrm{s.t.} \
V\cdot\vec{s}=\vec{b}, \ \vec{s}\ge 0\;,
\label{seibold:eq:linear_minimization}
\end{equation}
where the weights $w_i=w(\|\vec{x}_i-\vec{x}_0\|_2)$ are defined by an appropriately
decaying (i.e.~$w(\delta) = o(|\delta|^{-2})$, see \cite{Seibold2008})
non-negative distance weight function $w$.
Problem \eqref{seibold:eq:linear_minimization} is a linear program (LP) in standard
form. It is bounded, since we have imposed sign constraints and the weights $w_i$
are all non-negative.
\begin{thm}
If the polyhedron \eqref{seibold:eq:polyhedron} is non-void, then the linear
minimization approach \eqref{seibold:eq:linear_minimization} yields minimal positive
stencils.
\end{thm}
\begin{proof}
The sign constraints in \eqref{seibold:eq:linear_minimization} ensure that the selected
stencil is \emph{positive}. The existence of a \emph{minimal} solution is ensured by
the \emph{fundamental theorem of linear programming} \cite{Vanderbei2001}.
If the LP \eqref{seibold:eq:linear_minimization} has a solution, then it also has a
basic solution, in which at most $k$ of the $m$ stencil entries $s_i$ are different
from zero.
\end{proof}
In contrast to least squares methods, the LP approach yields non-zero values only for
a few selected points. This is in line with the work of
Donoho \cite{Donoho2006_1,Donoho2006_2}, who shows that a key advantage of
$\ell^1$ minimization is sparsity.
\begin{rem}
For regular grids, the linear minimization approach selects standard finite difference
stencils. For instance, for a regular Cartesian grid, the classical 5-point stencil (2d)
or 7-point stencil (3d), respectively, is obtained. In these cases, the basic solution is
degenerate, i.e.~some of the basis variables are zero.
Also for the configuration given in Example~\ref{seibold:ex:qm_nonpos_stencil}, the
linear minimization approach selects the classical 5-point Laplace stencil.
\end{rem}

\section{Existence of Positive Stencils and Matrix Connectivity}
\label{seibold:sec:lp_solution_existence_connectivity}
The question of the existence of positive stencils, i.e.~whether the
polyhedron \eqref{seibold:eq:polyhedron} is non-void, is nontrivial.
In \cite{Seibold2008}, the connection between the local point cloud geometry and the
existence of a positive stencil has been presented. The key idea is to consider the
dual linear program to \eqref{seibold:eq:linear_minimization} by using
Farkas' Lemma \cite{Vanderbei2001}. This yields a simple necessary criterion, as well as
a simple sufficient criterion, as follows.

\begin{thm}[Necessary Criterion for Positive Stencils]
\label{seibold:thm:mps_pos_stencil_necessary}
Let the considered central point w.l.o.g.~lie in the origin.
If a set of points $X\subset\mathbb{R}^d$ around the origin admits a positive
Laplace stencil, then they must not lie in one and the same half space (with
respect to an arbitrary hyperplane through the origin).
\end{thm}
\begin{thm}[Sufficient Criterion for Positive Stencils]
\label{seibold:thm:mps_pos_stencil_sufficient}
Let the considered central point w.l.o.g.~lie in the origin. Consider $d\in\{2,3\}$,
and let $\beta=\sqrt{2}-1$ in 2d, and $\beta=\sqrt{\frac{1}{6}(3-\sqrt{6})}$ in 3d.
Assume the set of points $X\subset\mathbb{R}^d$ has the property that in any
(i.e.~$\vec{v}$ with $\norm{\vec{v}}_2=1$ arbitrary) cone
$\vec{v}\cdot\vec{x}>\prn{1+\beta^2}^{-\frac{1}{2}}\|\vec{x}\|_2$,
at least one neighboring point $\vec{x}_i\in X$ is contained.
Then a positive Laplace stencil exists.
\end{thm}
We call this criterion \emph{cone criterion}.
In 2d, the cone has a total opening angle 45\ensuremath{^\circ}, and in 3d,
the total opening angle is 33.7\ensuremath{^\circ}.
Using the cone criterion, the existence of positive stencils can be guaranteed by
applying the $\ell^1$ minimization \eqref{seibold:eq:linear_minimization} to a
large enough candidate set of neighboring points.
As in \cite{Levin1998}, we define
\begin{defn}
\label{seibold:def:mesh_size}
Let $\Omega\subset\mathbb{R}^d$ be a domain and $X=\{\vec{x}_1,\dots,\vec{x}_n\}$
a point cloud. The \emph{mesh size} $h$ is defined as the minimal real number, such
that $\bar\Omega\subset \bigcup_{i=1}^n \bar B\prn{\vec{x}_i,\frac{h}{2}}$,
where $\bar B\prn{\vec{x},r}$ is the closed ball of radius $r$ centered in $\vec{x}$
and $\bar\Omega$ is the closure of $\Omega$.
\end{defn}
\begin{thm}[Condition on Point Cloud to Guarantee Positive Stencils]
\label{seibold:thm:pos_stencil_mesh_size}
Consider a point cloud of mesh size $h$, and let $\gamma$ be defined as in
Thm.~\ref{seibold:thm:mps_pos_stencil_sufficient}. If the radius of
considered candidate points satisfies $r > \frac{1}{\sin(\gamma/2)}\frac{h}{2}$,
then for every interior point which is sufficiently far from the boundary, a positive
stencil exists.
\end{thm}
The proof, and a description of special conditions near the boundary, are given
in \cite{Seibold2008}. The ratio of candidate radius to maximum hole size radius
$\frac{r}{h/2}$ equals 2.61 in 2d, and 3.45 in 3d.
In practice, one generally observes that already much smaller smaller ratios yield
positive stencils.

Due to Thm.~\ref{seibold:thm:ess_irred_is_ess_diagdom} and
Thm.~\ref{seibold:thm:Lmatrix_is_Mmatrix}, the matrix composed of positive stencils is
an M-matrix, if every interior and Neumann boundary point is connected to a Dirichlet
boundary point. As discussed in \cite{Seibold2008}, the connectivity of every point to
the boundary follows from Thm.~\ref{seibold:thm:mps_pos_stencil_necessary}, and the
connectivity to a Dirichlet point can be ensured under comparably mild technical
assumptions. Hence, if the implementation of the minimal positive stencil method ensures
that Dirichlet boundary points are used in the stencils of nearby points, then the
approach is guaranteed to generate M-matrices.

\section{Computational Effort to Generate the System Matrix}
\label{seibold:sec:effort_generation}
In meshfree methods, the actual generation of the linear system that approximates the
Poisson equation \eqref{seibold:eq:poisson_equation} carries a non-negligible
computational cost. For each interior point, a minimization problem has to be solved.

Classical least squares methods are based on the quadratic minimization
\eqref{seibold:eq:quadratic_minimization}.
The computation of the stencil by expression \eqref{seibold:eq:stencil_lsq}
requires first the set-up of the $k\times k$ matrix $VWV^T$, then the solution of the
the linear system $(VWV^T)\cdot\vec{v}=\vec{b}$, and lastly the computation of the
product $\vec{s} = WV^T\cdot\vec{v}$. This requires $k(k+1)m+\frac{k^3}{3}$ floating
point operations \cite[p.~150]{SeiboldDiss2006}.

In contrast, the linear minimization approach \eqref{seibold:eq:linear_minimization}
does not admit an explicit solution formula. Instead, a small linear program (LP) has
to be solved. To our knowledge, there are no general results about efficient methods
for such small LPs, especially any that would consider the special structure of the
Vandermonde matrix. A numerical comparison of various methods has been presented
in \cite[p.~148]{SeiboldDiss2006}. Simplex methods \cite{Chvatal1983} perform best for
the arising LPs. A basis change corresponds to one stencil point replacing another.
The theoretical worst case performance of simplex methods is not observed. Typical runs
find the solution in about $1.5 k$ steps, resulting in a complexity of $O(k^2m)$, which
equals the asymptotic effort of least squares approaches.

\section{Numerical Results}
\label{seibold:sec:numerics}
The two discretization approaches, least-squares vs.~minimal positive stencils, have
been compared in \cite{Seibold2008} in terms of accuracy of the finite difference
approximation. For a sequence of point clouds, the global truncation errors with both
approaches have been estimated numerically. The observation is that the approximation
errors of the two approaches differ by at most 20\%.
An interesting aspect is that the measured order of convergence is generally second
order, although the constraints \eqref{seibold:eq:constraints_laplace} enforce only
first order accuracy. This effect is also observed in other types of meshfree approaches.

Since the two approaches do not differ significantly in the accuracy of approximation,
a comparison of their computational effort when generating and solving the linear system
is of interest. The presented linear minimization approach generates optimally sparse
M-matrices, and we can hope that the solution of the arising linear systems is faster
than for systems that are generated with classical least squares approaches.

We consider two types of multigrid/multilevel approaches.
In Sect.~\ref{seibold:subsec:performance_samg}, we investigate the performance of
SAMG (``Algebraic Multigrid Methods for Systems'') \cite{SAMG}, by the
\emph{Fraunhofer Institute for Algorithms and Scientific Computing}.
It is a library of AMG subroutines, based on the code
RAMG05 \cite[App.~A]{TrottenbergOosterleeSchuller2001}, which is a successor of
the public domain code AMG1R5 \cite{RugeStueben1986}. SAMG is implemented as
a preconditioner for a BiCGstab \cite{VanDerVorst1992} iteration. It performs
a full multigrid cycle with Gauss-Seidel relaxation, and sparse Gaussian elimination
on the coarsest level. The numerical tests presented here are performed using
standard coarsening \cite[pp.~473]{TrottenbergOosterleeSchuller2001},
though aggressive coarsening strategies \cite[pp.~476]{TrottenbergOosterleeSchuller2001}
frequently turn out to improve the performance.
In Sect.~\ref{seibold:subsec:performance_amli}, we consider an AMLI type method,
implemented by C.~Mense and R.~Nabben, TU Berlin, 2005.

\subsection{Performance of SAMG}
\label{seibold:subsec:performance_samg}
For a 3d Poisson test problem, the classical least squares approach is compared to the
linear minimization approach in terms of the computational cost. We consider the geometry
shown in Fig.~\ref{seibold:fig_poisson3d_geometry}. It is motivated by the simulation of
the flow of a viscous material in a geometry with a thin outlet. More detail is given
in \cite[p.~183]{SeiboldDiss2006}. The considered Poisson
equation \eqref{seibold:eq:poisson_equation} has a smooth right hand side $f$. Neumann
boundary conditions are prescribed everywhere except for the thin outlet, at which
Dirichlet boundary conditions are prescribed. For the considered geometry, the classical
least squares approach is compared to the linear minimization approach in terms of the
computational cost. For a sequence of point clouds with increasing resolution, the
geometry is discretized. One approximation is generated using a weighted least squares
approach \eqref{seibold:eq:quadratic_minimization}, with about 40 neighbors for each
interior point. Another approximation is generated from the $\ell^1$ minimization approach
\eqref{seibold:eq:linear_minimization}. By construction, each interior point has only
10 neighbors. Consequently, the $\ell^1$ matrices possess about 4 times fewer non-zero
entries than the LSQ matrices.

\begin{figure}
\centering
\begin{minipage}[t]{.7\textwidth}
\centering
\includegraphics[width=.86\textwidth]{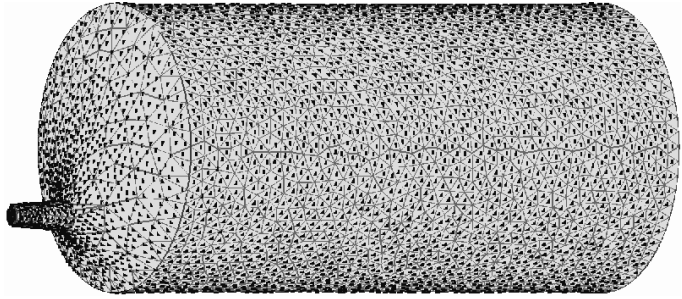}
\vspace{-.3em}
\caption{Geometry for the 3d Poisson test problem}
\label{seibold:fig_poisson3d_geometry}
\end{minipage}
\end{figure}

\begin{figure}
\centering
\begin{minipage}[t]{.49\textwidth}
\centering
\includegraphics[width=0.9\textwidth]{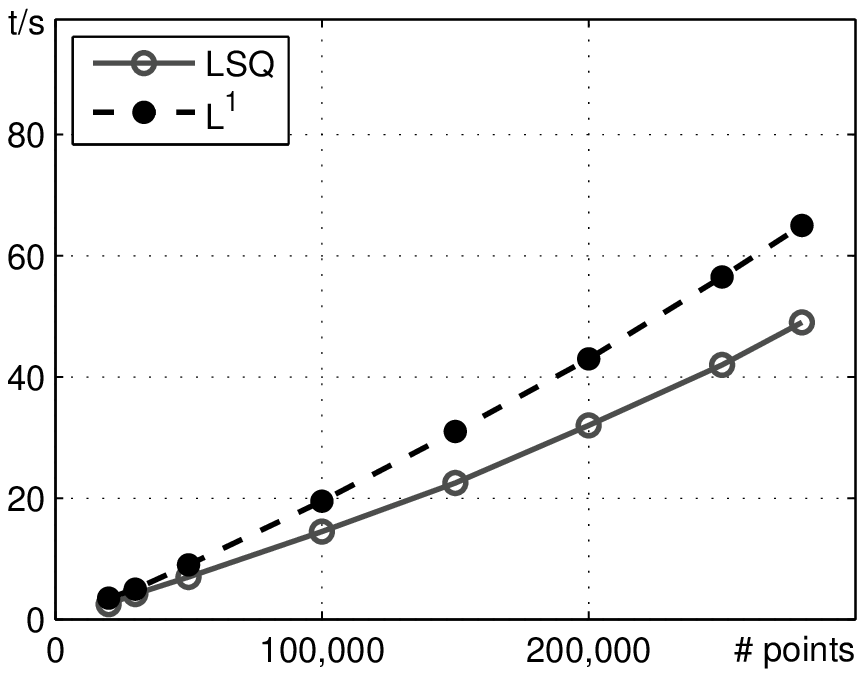}
\vspace{-.5em}
\caption{CPU times for setup}
\label{seibold:fig_cputimes_setup}
\end{minipage}
\hfill
\begin{minipage}[t]{.49\textwidth}
\centering
\includegraphics[width=0.9\textwidth]{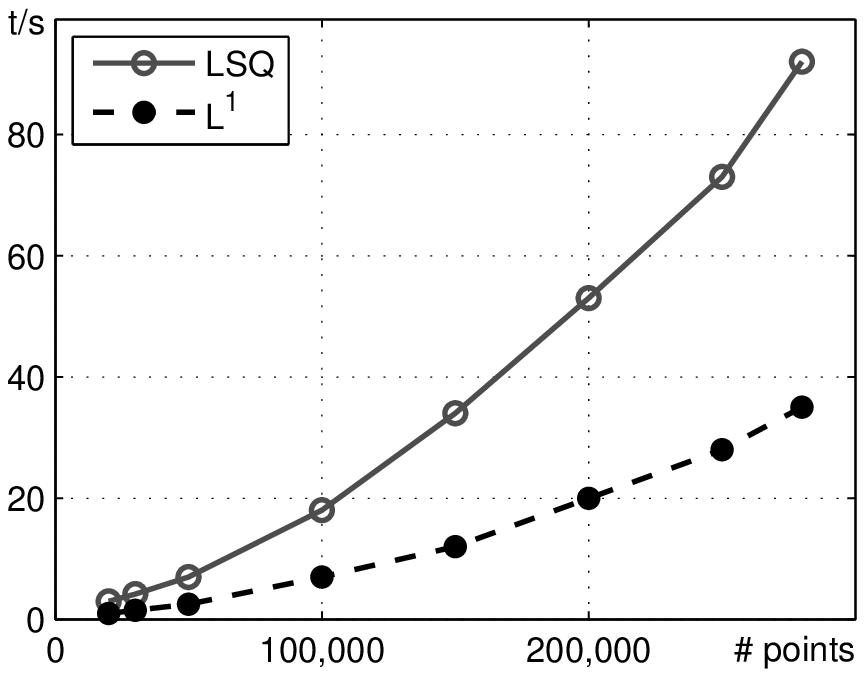}
\vspace{-.5em}
\caption{CPU times for solve with BiCGstab}
\label{seibold:fig_cputimes_solve_bicg}
\end{minipage}

\vspace{1em}

\begin{minipage}[t]{.49\textwidth}
\centering
\includegraphics[width=0.9\textwidth]{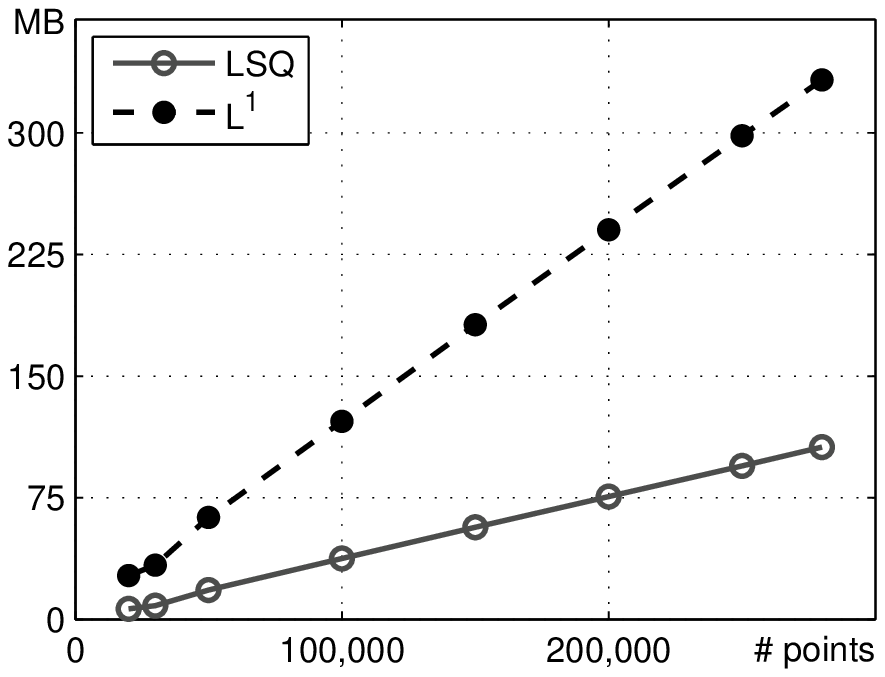}
\vspace{-.5em}
\caption{Memory consumption with SAMG}
\label{seibold:fig_cputimes_memory}
\end{minipage}
\hfill
\begin{minipage}[t]{.49\textwidth}
\centering
\includegraphics[width=0.9\textwidth]{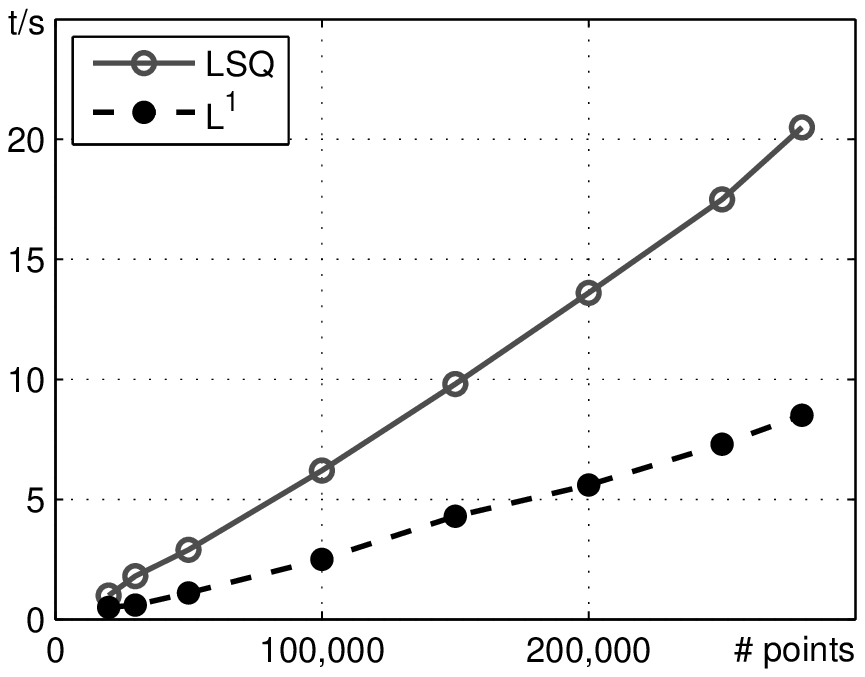}
\vspace{-.5em}
\caption{CPU times for solve with SAMG}
\label{seibold:fig_cputimes_solve_amg}
\end{minipage}
\end{figure}

Fig.~\ref{seibold:fig_cputimes_setup} shows the CPU times for the setup of the system
matrices. The solid curve represents the costs with the least squares approach,
which is based on expression  \eqref{seibold:eq:stencil_lsq}. The dashed curve shows
the cost of the linear optimization approach, for which the simplex method is used.
As predicted in Sect.~\ref{seibold:sec:effort_generation}, asymptotically the two approaches
are equally expensive. In the considered example, the linear optimization approach
is about 20\% more expensive. Note that the implemented simplex algorithm is not
optimized. Better linear programming methods, and/or smart heuristics
may accelerate the matrix setup phase.

Fig.~\ref{seibold:fig_cputimes_solve_bicg} shows the CPU times for the solution of the
arising system with a BiCGstab method \cite{VanDerVorst1992},
preconditioned by ILU(0) \cite[Chap.~10]{Saad1996}.
In comparison, Fig.~\ref{seibold:fig_cputimes_solve_amg} shows the CPU times for the
solution of the same system with SAMG \cite{SAMG}.
Three aspects can be observed:
\begin{itemize}
\item
The cost of the BiCGstab method grows superlinearly with the number of unknowns.
In contrast, the SAMG cost is almost perfectly linear in the number of points.
\item
The SAMG method solves the arising linear systems significantly faster than the
BiCGstab method. For low resolutions, SAMG is 2-3 times as fast as BiCGstab.
For high resolutions, it is 4-5 times as fast.
While the used BiCGstab method is certainly not fully optimized, the performance
of SAMG is still compelling.
\item
For both solution methods, the linear systems that arise from the $\ell^1$ minimization
approach are 3-4 times faster to solve than the systems generated by the LSQ approach.
\end{itemize}
Fig.~\ref{seibold:fig_cputimes_memory} shows the memory consumption in an SAMG solve.
As one can expect, it is perfectly linear in the number of unknowns, and proportional
to the number of non-zero matrix entries.

The first conclusion we draw from this numerical experiment is that SAMG performs very
well on both LSQ matrices and $\ell^1$ matrices. In particular, it is always observed to
converge, even though the code is primarily designed for symmetric matrices, and even
though most of the considered LSQ matrices are not M-matrices. SAMG is significantly
faster than a classical BiCGstab method with ILU(0) preconditioning. In addition, the
computational effort of SAMG is linear in the number of unknowns.

The second conclusion is that the matrices obtained from the $\ell^1$ minimization
approach are solved significantly faster than classical LSQ matrices. The observed
speed-up equals the increase in sparsity that the $\ell^1$ minimization approach yields.
In another test (not shown here), we store the non-basis neighbors, as obtained by the
simplex method, as actual zero-entries in the matrix. The solution times with those
matrices turn out to be within 10\% of the solution times with the LSQ matrices.
This implies that the M-matrix property itself does not increase the actual convergence
rate of linear solvers. On the other hand, very sparse Poisson stencils are often
times observed to be less robust than larger ones. In this sense, one can interpret the
M-matrix property as an ingredient that allows the selection of a minimal stencil,
without actually worsening convergence properties of linear solvers.

\subsection{Performance of an AMLI Type Method}
\label{seibold:subsec:performance_amli}
We consider a 2d Poisson test problem on a circular domain
$\Omega = \{\vec{x}\in\mathbb{R}^2:\|\vec{x}\|_2<1\}$ as
\begin{equation*}
\begin{cases}
  -\Delta u = f &\mathrm{in~}\Omega \\
\quad\;\; u = f &\mathrm{on~}\partial\Omega
\end{cases}
\end{equation*}
where $f(x,y) = \sin(4x+0.1)+x\,\cos(2y+0.4)$.
The analytical solution is $u(x,y) = 16\sin(4x+0.1)+4x\,\cos(2y+0.4)$.
This problem is purposefully chosen to have a simple geometry, so that the focus lies
on effects due to the unstructured point cloud. In \cite[p.~173]{SeiboldDiss2006},
this problem is used to analyze discretization errors on various point clouds.
Here, we restrict to one specific point cloud, shown in
Fig.~\ref{seibold:fig_poisson2d_cloud}. Boundary points are placed equidistantly
on the unit circle. In the interior, 4000 points are placed.
For the given point cloud, we generate four
types of discretization matrices: the $\ell^1$ minimization matrix, and three LSQ
matrices, with the closest 5 / 12 / 36 neighboring points considered in each stencil.
For each matrix, the linear system is solved by various versions of AMLI type two-grid
methods, all implemented by C.~Mense and R.~Nabben, TU Berlin, 2005.

Two versions of coarsening are considered.
In \emph{default coarsening}, the first half of the unknowns is chosen as coarse scale
variables, and the other ones are selected as fine scale variables. Since here the
points are not sorted, this is equivalent to a random selection of coarse variables.
In the \emph{Ruge-St{\"u}ben coarsening} \cite{RugeStueben1987}, the coarsening of the
matrix graph is performed preferably in the directions of large couplings, where the
coupling strength (for an M-matrix) is defined as the magnitude of the corresponding
off-diagonal entries.

Four versions of AMLI type methods are considered, which are based on different
splittings \cite{MenseNabben2008_1,MenseNabben2008_2}.
Specifically, besides the standard additive AMLI, a multiplicative (MAMLI),
a reverse multiplicative (RMAMLI), and a symmetrized multiplicative (SMAMLI) splitting
are considered.

\begin{figure}
\hspace{-1.3cm}
\begin{minipage}[t]{.44\textwidth}
\centering
\includegraphics[width=.70\textwidth]{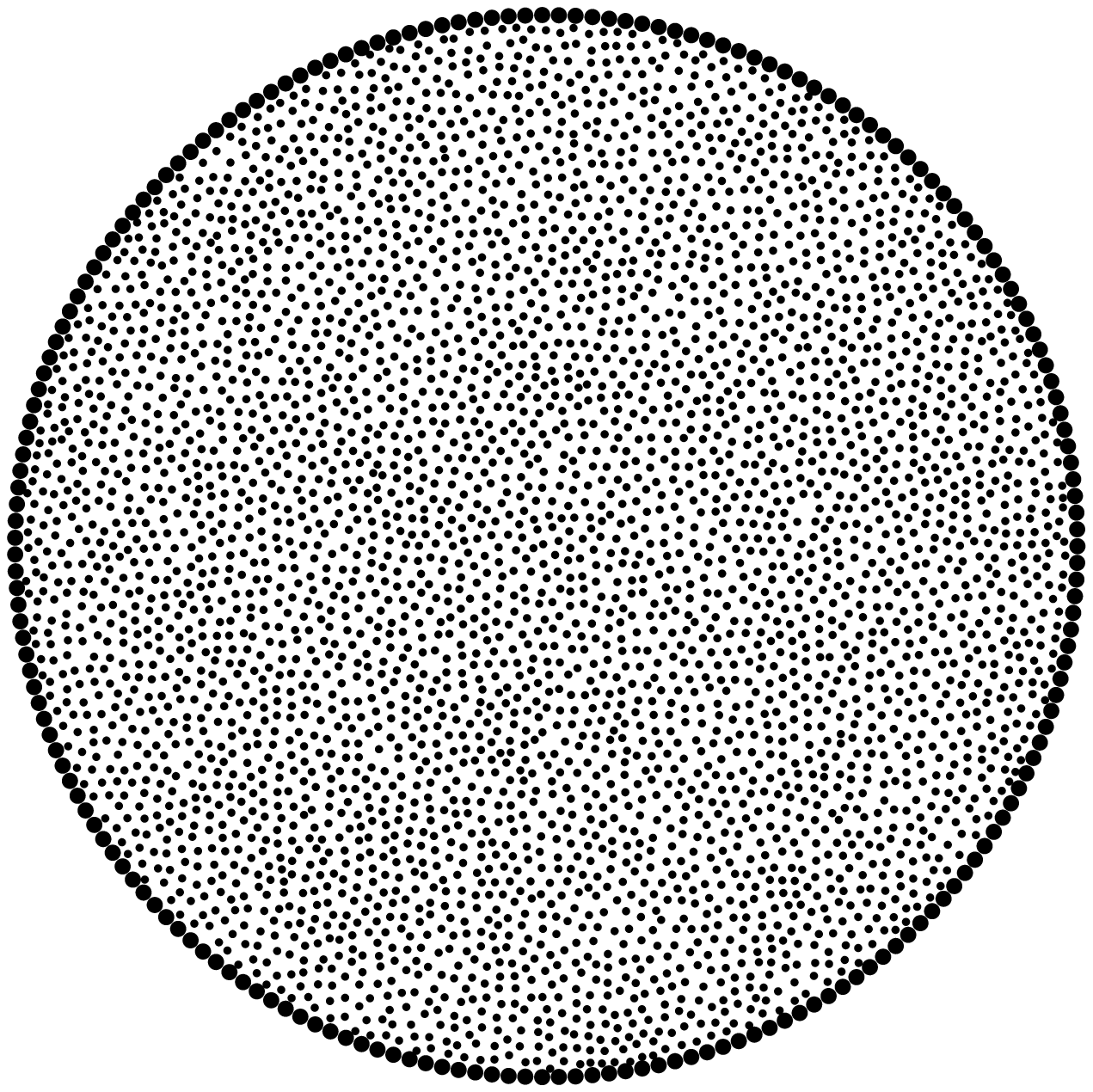}
\caption{Point cloud on a circular domain}
\label{seibold:fig_poisson2d_cloud}
\end{minipage}
\hfill
\begin{minipage}[t]{.63\textwidth}
\centering
\includegraphics[width=.99\textwidth]{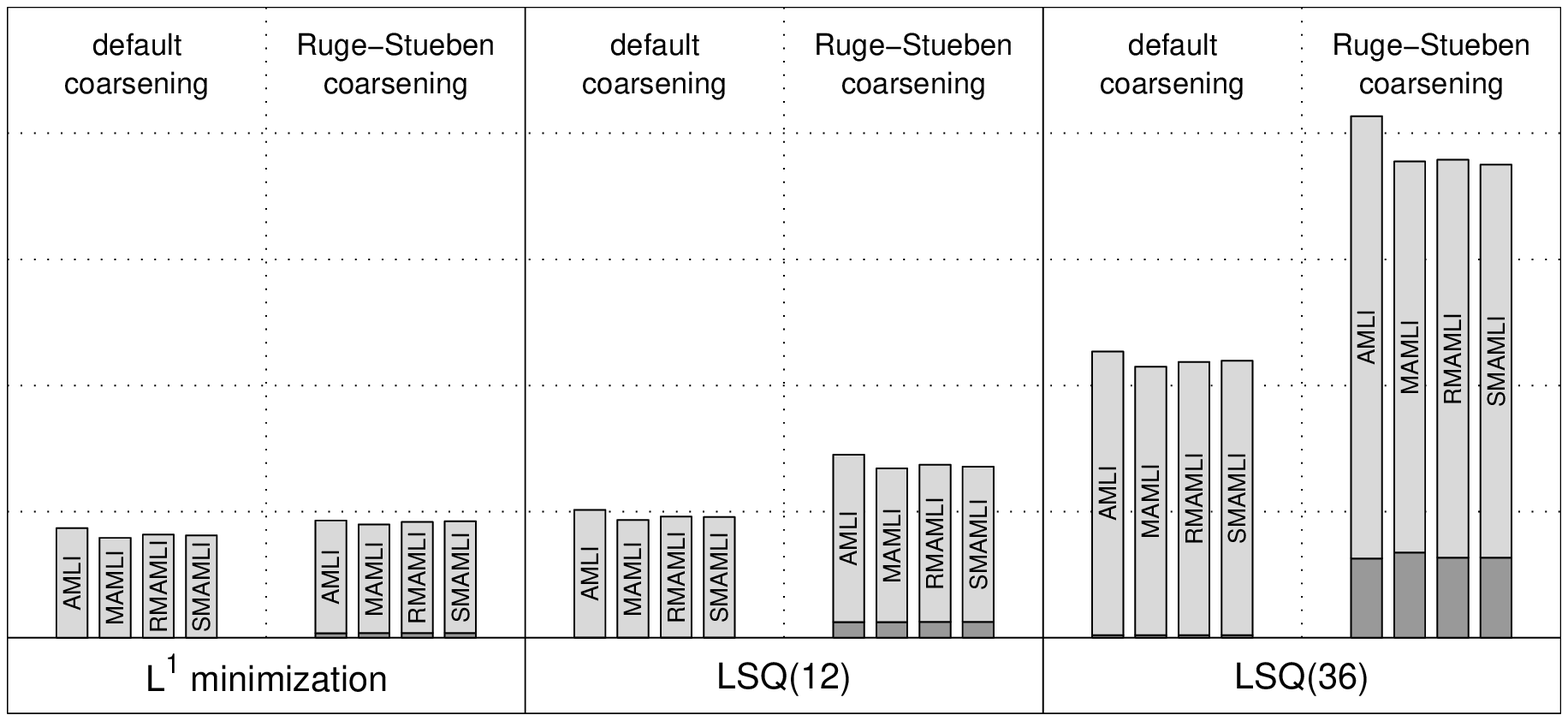}
\caption{Relative CPU times for various types of AMLI approaches}
\label{seibold:fig_amli_times}
\end{minipage}
\end{figure}

Fig.~\ref{seibold:fig_amli_times} shows the relative\footnote{The absolute CPU times
are not representative here and are thus omitted.} CPU times for the various test
cases. The left third corresponds to the $\ell^1$ matrix, the middle third shows the
LSQ approach with 12 neighbors for each point, and the right third represents
the LSQ approach with 36 neighbors. The CPU times for the approach with the 5 nearest
neighbors is not shown, since the AMLI methods failed to yield the correct solution.
Needless to say, the obtained matrices are not M-matrices. However, also for the
LSQ(12) and the LSQ(16) case, the Poisson matrices are in general not M-matrices.
Nevertheless, the AMLI methods converge. In each case, the height of the dark part
of the bar shows the setup time, and the light part shows the actual solution time.

One can observe that the $\ell^1$ matrices lead to the shortest solution times.
However, unlike with the SAMG method, the speed-up factor with respect to the LSQ(12)
case is not significant, although the $\ell^1$ matrix is only half as dense. With
respect to the LSQ(36) matrices, a noticeable speed-up is observed, although again not
as significant as one would expect from the sparsity factors. Another observation is
that for the considered example, there is no significant difference between the
various splitting versions. Finally, for larger stencils, the Ruge-St{\"u}ben
coarsening is more expensive than the default coarsening. While this could be
expected for the setup phase, it is surprising for the iteration phase.
We cannot provide a satisfying explanation for this observation.

\section{Conclusions and Outlook}
\label{seibold:sec:conclusions_outlook}
With the pressure correction step in particle methods for incompressible fluid flows,
we have presented an application that leads to large, sparse, non-symmetric matrices.
These matrices that approximate the Poisson equation are constructed using a
meshfree finite difference method. They have a different structure from traditionally
considered non-symmetric matrices, such as Markov chain representation matrices.
In particular, classical meshfree least squares approaches yield in general matrices
that do not have an M-matrix structure.

However, the M-matrix property is a key ingredient in recent convergence proofs of
AMG and AMLI type methods for non-symmetric matrices. While it is not a necessary
condition, convergence is not ensured if the M-matrix property is violated. In particle
methods, Poisson equations are solved in every time step with ever changing geometries,
and one single failure of convergence could spoil a whole long-term computation.
Therefore, until multigrid convergence can be shown under weaker assumptions, the
M-matrix structure is a property that should be guaranteed. The presented approach
does so. Using an $\ell^1$ minimization formulation, and linear programming methods,
optimally sparse M-matrices are generated. Conditions are provided that ensure the
success of this approach.

The performance of multigrid methods for meshfree finite difference approaches is
numerically investigated in various test cases. With SAMG, a sophisticated AMG method
is applied, and with AMLI, an approach is used in which convergence for M-matrices is
proved. The main conclusion is that the $\ell^1$ minimization approach guarantees
convergence and generates matrices that are significantly sparser than classical LSQ
matrices, and this sparsity results in a speed-up in multigrid solvers.

Currently, the setup of the system matrices with the $\ell^1$ minimization approach
is slightly slower than with LSQ approaches. However, more efficient linear programming
methods may turn the tide towards the linear optimization approach. The efficient solution
of the linear programs is a crucial point in the presented method, and worth a deeper
analysis. Another interesting question is whether a weaker requirement than the
M-matrix property already implies convergence of AMG and AMLI methods.
Numerical tests presented in \cite{SeiboldDiss2006} indicate that the convergence
of AMLI methods tends to be related to the inverse positivity of an LSQ matrix,
rather than an M-matrix structure. In fact, LSQ matrices are frequently observed
to be inverse positive, while not having an L-matrix structure. Alternative
characterizations of inverse positive matrices \cite{FujimotoRanade2004} may help
answer this question.

\section*{Acknowledgments}
The author would like to thank Dr.~J{\"o}rg Kuhnert, Dr.~Christian Mense,
Dr.~Klaus St{\"u}ben, and Dr.~Sudarshan Tiwari
for fruitful discussions and support on meshfree finite difference methods
and on the application of algebraic multigrid codes.
The support by the National Science Foundation is acknowledged.
The author was partially supported by NSF grant DMS--0813648.

\bibliographystyle{plain}
\bibliography{references_complete}

\end{document}